\documentclass[12pt]{amsart}
\usepackage{amssymb,amsmath}
\usepackage{geometry}    
\usepackage{mathrsfs}

\usepackage{graphicx}

\usepackage{vmargin}
\setmargrb{1in}{1in}{1in}{1in}

\newtheorem{theorem}{Theorem}[section]

\newtheorem{proposition}[theorem]{Proposition}
\newtheorem{corollary}[theorem]{Corollary}
\theoremstyle{definition}

\begin{document}

\title[Geometry of numbers and uniform
approximation]{On geometry of numbers and uniform rational approximation to the Veronese curve}

\author{Johannes Schleischitz}


\thanks{Middle East Technical University, Northern Cyprus Campus, Kalkanli, G\"uzelyurt \\
	jschleischitz@outlook.com, johannes@metu.edu.tr}


\begin{abstract}
	Consider the classical problem of rational simultaneous approximation
	to a point in $\mathbb{R}^{n}$. The optimal lower bound on the gap between the induced ordinary and uniform
	approximation exponents has been established by Marnat and Moshchevitin
	in 2018. Recently Nguyen, Poels and Roy provided information on the best
	approximating rational vectors to points where the gap is close
	to this minimal value. Combining the latter result 
	with parametric geometry of numbers, we effectively bound the dual linear form exponents in the described situation.
	As an application, we slightly improve the upper bound for the classical exponent of uniform Diophantine approximation $\widehat{\lambda}_{n}(\xi)$, for even $n\geq 4$. Unfortunately our improvements are small, for $n=4$
	only in the fifth decimal digit. However, the 
	underlying method in principle can be improved with more effort
	to provide better bounds. We indeed establish reasonably stronger results 
	for numbers which almost satisfy equality in the estimate by Marnat and Moshchevitin. We conclude with consequences on the classical problem of approximation to real numbers by algebraic numbers/integers of uniformly bounded degree.
\end{abstract}

\maketitle

{\footnotesize{

{\em Keywords}: exponents of Diophantine approximation, regular graph,
parametric geometry of numbers\\
Math Subject Classification 2010: 11J13, 11J82}}

\vspace{1mm}

\section{Consequences of a recent Theorem in geometry of numbers} \label{sek1}

Classical topics in Diophantine approximation are to study
simultaneous rational approximation to points $\underline{\xi}=(\xi_{1},\ldots,\xi_{n})\in\mathbb{R}^{n}$
and to find small values of linear forms in $\underline{\xi}$ with integer coefficients.
This naturally leads to investigation
of classical exponents of Diophantine approximation.
Denote by $\omega(\underline{\xi})$ the (possibly infinite) supremum
of $\omega$ such that
\begin{equation}  \label{eq:rihs}
Y_{\underline{x}}:= \max_{1\leq i\leq n} |x\xi_{i}-y_{i}| \leq x^{-\omega}
\end{equation}
has infinitely many integer vector solutions $\underline{x}=(x,y_{1},\ldots,y_{n})\in\mathbb{Z}^{n+1}$.
This is usually referred to as the ordinary exponent of simultaneous
approximation. Similarly, let the (possibly infinite) uniform exponent of simultaneous
approximation denoted by $\widehat{\omega}(\underline{\xi})$ be 
the supremum of $\omega$ such that
\begin{equation}  \label{eq:bru}
0<x\leq X, \qquad Y_{\underline{x}} \leq X^{-\omega}
\end{equation}
has a solution $\underline{x}\in\mathbb{Z}^{n+1}$ for all large $X$.
For linear forms, let $\omega^{\ast}(\underline{\xi})$ and
$\widehat{\omega}^{\ast}(\underline{\xi})$ respectively
be the (possibly infinite) supremum of $\omega^{\ast}$ such that
\[
0<\max_{1\leq j\leq n} |a_{j}|\leq X, \qquad |a_{0}+a_{1}\xi_{1}+\cdots+a_{n}\xi_{n}| \leq X^{-\omega^{\ast}}
\]
has a solution in integers $a_{j}$ for certain arbitrarily large $X$ and all large $X$, respectively. Variants of Dirichlet's Box principle yield the lower
bounds
\begin{equation} \label{eq:tri}
\omega(\underline{\xi})\geq \widehat{\omega}(\underline{\xi})\geq \frac{1}{n}, \qquad \omega^{\ast}(\underline{\xi})\geq \widehat{\omega}^{\ast}(\underline{\xi})\geq n,
\end{equation}
for any $\underline{\xi}\in\mathbb{R}^{n}$. 
Finding refined relations between these exponents is an
important topic in Diophantine approximation. 
Assume in the sequel that $\underline{\xi}$
is linearly independent over $\mathbb{Q}$ together with $\{1\}$.
Marnat and Moshchevitin~\cite{mamo} recently proved a remarkable (sharp)
improvement of the trivial inequalities $\omega(\underline{\xi})\geq \widehat{\omega}(\underline{\xi})$ and $\omega^{\ast}(\underline{\xi})\geq \widehat{\omega}^{\ast}(\underline{\xi})$
that had been conjectured
by Schmidt and Summerer~\cite{ssmj}. The first can be
stated as
\begin{equation}  \label{eq:mar}
\widehat{\omega}(\underline{\xi})+
\frac{\widehat{\omega}(\underline{\xi})^{2}}{\omega(\underline{\xi})}+\cdots+
\frac{\widehat{\omega}(\underline{\xi})^{n}}{\omega(\underline{\xi})^{n-1}} \leq 1,
\end{equation}
where we mean that $\widehat{\omega}(\underline{\xi})=1$ implies
$\omega(\underline{\xi})=\infty$. The sharpness of the estimate
had previously been settled by Roy who constructed the equality case for any possible parameter pair in~\cite{damien1}.
We further want to refer to his reasonably more general,
deep results in~\cite{damien2}.  
We do not require the dual linear form estimate in this paper, but 
rather want to
refer to~\cite{ssmh},~\cite{gm} for a weaker inequality that had been established previously, and the PhD thesis of
Rivard-Cooke~\cite{rivart} for a simplified proof of \eqref{eq:mar}
and a conjectural generalization.
Equality in \eqref{eq:mar} is obtained in a case
that Schmidt and Summerer in~\cite{ssmj} refer to
as the {\em regular graph}. We omit the geometrical motivation behind it, 
but point out that in this scenario the logarithms of 
the numbers $x$ realizing the exponent $\omega(\underline{\xi})$ 
in \eqref{eq:rihs} (for appropriate $y_{i}$)
form almost a geometric sequence with ratio $\omega(\underline{\xi})/\widehat{\omega}(\underline{\xi})$. 
The recent preprint by Nguyen, Poels and Roy~\cite{vpr} 
provides a clearer
picture by essentially showing that if $\omega(\underline{\xi}), \widehat{\omega}(\underline{\xi})$ 
almost satisfy identity in the estimate \eqref{eq:mar},
then similar properties
must still be satisfied. Indeed this information is provided in a compact form 
in Theorem~1.4 from~\cite{vpr}, which we rephrase below upon omitting 
some subtle additional information
not required here. With $Y_{\underline{x}}$ defined in \eqref{eq:rihs},
for a parameter $X>1$ we derive the quantity
\[
\mathscr{L}_{\underline{\xi}}(X)= \min \{ Y_{\underline{x}}: 1\leq x\leq X  \},
\]
where the minimum is taken over all integer vectors $\underline{x}=(x,y_{1},\ldots,y_{n})$ with 
first coordinate positive and not exceeding $X$. The induced sequence of vectors
$\underline{x}$ realizing the minimum for some $X$
are sometimes referred to as best approximations or minimal points.

\begin{theorem}[Nguyen, Poels, Roy] \label{v}
	Let $n\geq 1$ be an integer.
	Let $\underline{\xi}\in\mathbb{R}^{n}$ with $\mathbb{Q}$-linearly independent coordinates together with $\{1\}$. Suppose that there exist positive real numbers $a,b,\alpha,\beta$ such that for all large enough $X$ we have
	\begin{equation}  \label{eq:yesc}
	bX^{-\beta} \leq \mathscr{L}_{\underline{\xi}}(X)\leq aX^{-\alpha}.
	\end{equation}
	Then $\alpha\leq \beta$ and 
	\begin{equation} \label{eq:eps2}
	\epsilon= \epsilon_{\alpha,\beta}:= 1-\left(\alpha+\frac{\alpha^{2}}{\beta}+\cdots+\frac{\alpha^{n}}{\beta^{n-1}} \right) \geq 0.
	\end{equation}
	If
	\begin{equation} \label{eq:epsilon}
	\epsilon \leq \frac{1}{4n} \left(\frac{\alpha}{\beta} \right)^{n} \min\{ \alpha, \beta-\alpha\},
	\end{equation}
	then there exists an unbounded sequence of integer vectors
	$\underline{x}_{j}=(x_{j},y_{j,1},\ldots,y_{j,n})$ 
	with the properties
	\begin{itemize}
		\item $| \alpha \log x_{j+1} - \beta \log x_{j}|\leq C+4\epsilon (\beta/\alpha)^{n} \log x_{j+1}$
		\item $| \log Y_{\underline{x}_{j}}+ \beta \log x_{j}| \leq C+4\epsilon (\beta/\alpha)^{2} \log x_{j}$
		\item the vectors $\underline{x}_{j}, \underline{x}_{j+1}, \ldots, \underline{x}_{j+n}$ are linearly independent
		\item There is no vector $\underline{x}=(x,y_{1},\ldots,y_{n})$
		so that $1\leq x\leq x_{j}$ and $Y_{\underline{x}}< Y_{\underline{x}_{j}}$
	\end{itemize}
\end{theorem}

Clearly assumption \eqref{eq:yesc} implies 
\[
\alpha\leq \widehat{\omega}(\underline{\xi})\leq \omega(\underline{\xi})\leq \beta,
\]
and conversely any $\underline{\xi}$ satisfies \eqref{eq:yesc} for any 
$\alpha<\widehat{\omega}(\underline{\xi})$ and $\beta>\omega(\underline{\xi})$
with suitable $a,b$. We will thus occasionally
identify $\alpha$ and $\beta$ with
$\widehat{\omega}(\underline{\xi})$ and $\omega(\underline{\xi})$,
respectively.
Then $\epsilon=0$ corresponds to equality in \eqref{eq:mar}.
In the situation of Theorem~\ref{v},
we use the description of best approximating vectors
to provide estimates on the dual
exponents $\omega^{\ast}(\underline{\xi}), \widehat{\omega}^{\ast}(\underline{\xi})$.

\begin{theorem} \label{new}
	Let $n\geq 1$ be an integer.
	Assume $\underline{\xi}=(\xi_{1},\ldots,\xi_{n})$ is $\mathbb{Q}$-linearly independent coordinates together with $\{1\}$ and satisfies condition \eqref{eq:yesc} of Theorem~\ref{v} with
	given $a,b,\alpha,\beta$. 
	Assume $\epsilon$ as in \eqref{eq:eps2} satisfies \eqref{eq:epsilon},
	and let
	\[
	\phi:= \frac{4\epsilon \beta^{n-1}}{\alpha^{n}}, \qquad \rho:= \frac{4\epsilon \beta^{2}}{\alpha^{2}}.
	\]
	Then we have
	\begin{equation} \label{eq:01}
	\widehat{\omega}^{\ast}(\underline{\xi}) \geq \frac{(\beta-\rho)S}{(\frac{\alpha}{\beta}-\phi)^{-n}+(\beta-\rho)(1-S)}, 
	\qquad S:= \sum_{j=1}^{n} (\frac{\alpha}{\beta}+\phi)^{1-j}.
	\end{equation}
	Moreover
	\begin{equation} \label{eq:03}
	\omega^{\ast}(\underline{\xi}) \geq
	\frac{\rho^{2}-\beta^{2}-(\beta+\rho)^{2}T}{\rho-\beta+(\beta+\rho)^{2}T},
	\qquad\qquad T:= \sum_{j=1}^{n-1}(\frac{\alpha}{\beta}+\phi)^{j}.
	\end{equation}
	Conversely we have the upper bounds 
	\begin{equation} \label{eq:wobene}
	\widehat{\omega}^{\ast}(\underline{\xi}) \leq (\beta-\rho)^{-1}(\frac{\alpha}{\beta}-\phi)^{-n},
	\end{equation}
	and
	\begin{equation} \label{eq:wobene2}
	\omega^{\ast}(\underline{\xi}) \leq (\beta-\rho)^{-1}(\frac{\alpha}{\beta}-\phi)^{-n-1}.
	\end{equation}
\end{theorem}

All bounds in Theorem~\ref{new}
can in principle be improved as will be indicated
in its proof, thereby
implying better bounds in Theorem~\ref{s2} below 
as well. However we do not
attempt to optimize the method as it
would lead to a significantly more cumbersome proof.
Theorem~\ref{new} becomes interesting for $n\geq 3$. For
$n=1$ clearly $\omega^{\ast}(\xi)=\omega(\xi)$ and a short argument by
Khintchine~\cite{kredi} shows $\widehat{\omega}(\xi)=\widehat{\omega}^{\ast}(\xi)=1$ for any $\xi\in\mathbb{R}\setminus \mathbb{Q}$.
For $n=2$, as soon as $1,\xi_{1},\xi_{2}$ are linearly
independent over $\mathbb{Q}$, we have Jarn\'ik's identity~\cite{vja} given as $\widehat{\omega}^{\ast}(\xi_{1},\xi_{2})=
(1-\widehat{\omega}(\xi_{1},\xi_{2}))^{-1}$
and the sharp estimates
\[
 \frac{\omega(\xi_{1},\xi_{2})+\widehat{\omega}(\xi_{1},\xi_{2})}{1-\widehat{\omega}(\xi_{1},\xi_{2})}\leq 
\omega^{\ast}(\xi_{1},\xi_{2}) \leq \frac{\omega(\xi_{1},\xi_{2})}{\widehat{\omega}(\xi_{1},\xi_{2})-\omega(\xi_{1},\xi_{2})+\omega(\xi_{1},\xi_{2})\widehat{\omega}(\xi_{1},\xi_{2})}
\]
due to Laurent~\cite{lauu}, where the right inequality
requires a positive right hand side (which is true upon
equality in \eqref{eq:mar}) and otherwise $\omega^{\ast}(\xi_{1},\xi_{2})=\infty$ is possible.
For additional information
in the case $n=2$, we refer to Roy~\cite{roy01} and Schmidt and Summerer~\cite{ssmat}. 
Concretely, our classical exponents
can be interpreted as extremal values of the first and last
successive minimum of some
lattice point problem from the geometry of numbers. The papers~\cite{roy01,ssmat} provide additional information 
on the missing second successive minimum, not reflected by $\omega, \widehat{\omega}, \omega^{\ast}, \widehat{\omega}^{\ast}$.
See also the introduction of Section~\ref{prof} for the 
correspondence geometry of numbers.
For larger $n$, sharp estimates linking $\omega(\underline{\xi})$ and $\omega^{\ast}(\underline{\xi})$
on the one hand and $\widehat{\omega}(\underline{\xi})$ and $\widehat{\omega}^{\ast}(\underline{\xi})$ on the
other hand are due to Khintchine (see Theorem~B5 in~\cite{bugbuch}) and German~\cite{german} respectively,
however they do not take into account the additional information 
in Theorem~\ref{new} on another exponent. As a consequence,
they do not provide any upper bound for $\omega^{\ast}(\underline{\xi})$ (or $\widehat{\omega}^{\ast}(\underline{\xi})$) if $\omega(\underline{\xi})\geq 1/(n-1)$ (or $\widehat{\omega}(\underline{\xi})\geq 1/(n-1)$), and indeed
these exponents may take the value $+\infty$ when no further restriction
is imposed. See however~\cite{buglau} for refinements of Khintchine's
estimates that contain triples $(\omega,\widehat{\omega},\omega^{\ast})$ resp. $(\omega, \omega^{\ast}, \widehat{\omega}^{\ast})$ of classical exponents.

We also should point out that Theorem~\ref{new} is mainly of interest when
$\epsilon$ in \eqref{eq:eps2} is very small, when we increase $\epsilon$
the inequalities \eqref{eq:01}, \eqref{eq:03} quickly 
become worse than the bounds in \eqref{eq:tri}.
Consequently, the numerical improvements in Theorem~\ref{s2} below will
be small.

A consequence of Theorem~\ref{new} is that equality in
\eqref{eq:mar} is sufficient for all classical
exponents to attain the values as in the corresponding regular graph, and
by continuity reasons they cannot differ much from them if the error 
in \eqref{eq:mar} is sufficiently small.

\begin{corollary}  \label{dadcor}
	Assume $\underline{\xi}=(\xi_{1},\ldots,\xi_{n})$ induces equality in \eqref{eq:mar}. Then 
	\[
	\widehat{\omega}^{\ast}(\underline{\xi})=\frac{\omega(\underline{\xi})^{n-1}}{\widehat{\omega}(\underline{\xi})^{n}},
	\qquad \omega^{\ast}(\underline{\xi})=\frac{\omega(\underline{\xi})^{n}}{\widehat{\omega}(\underline{\xi})^{n+1}}.
	\]
	Moreover, for any $\varepsilon>0$ and $c<1$, if
	we identify $\alpha=\widehat{\omega}(\underline{\xi})$ and $\beta=\omega(\underline{\xi})$, upon $\alpha\leq c$ 
	 there exists $\delta=\delta(n,c)>0$ such that the estimate
	$\epsilon=\epsilon_{\alpha,\beta}<\delta$ implies 
	\[
	\frac{\omega(\underline{\xi})^{n-1}}{\widehat{\omega}(\underline{\xi})^{n}}-\varepsilon\leq	\widehat{\omega}^{\ast}(\underline{\xi})\leq \frac{\omega(\underline{\xi})^{n-1}}{\widehat{\omega}(\underline{\xi})^{n}}+\varepsilon,
	\qquad \frac{\omega(\underline{\xi})^{n}}{\widehat{\omega}(\underline{\xi})^{n+1}}-\varepsilon\leq \omega^{\ast}(\underline{\xi})\leq \frac{\omega(\underline{\xi})^{n}}{\widehat{\omega}(\underline{\xi})^{n+1}}+\varepsilon.
	\]
\end{corollary}

Unfortunately $\delta$ depends in a sensitive way on $\epsilon$
(that is, on $\alpha,\beta$), concordant with the strong dependence
of the bounds in Theorem~\ref{new} on $\epsilon$.

\section{An application to the Veronese curve} \label{vc}

We consider $\underline{\xi}=(\xi,\xi^{2},\ldots,\xi^{n})$ on the Veronese
curve, where we assume $\xi$ is transcendental. As customary we write
\begin{equation}  \label{eq:clex}
\lambda_{n}(\xi)= \omega(\underline{\xi}), \quad
\widehat{\lambda}_{n}(\xi)= \widehat{\omega}(\underline{\xi}), \quad
w_{n}(\xi)= \omega^{\ast}(\underline{\xi}), \quad
\widehat{w}_{n}(\xi)= \widehat{\omega}^{\ast}(\underline{\xi}),
\end{equation}
for the intensely studied classical exponents of Diophantine approximation
(note: the classical definition of $w_{n}(\xi), \widehat{w}_{n}(\xi)$ differs if $\xi$ is algebraic of degree $\leq n$, by excluding vanishing polynomial evaluations at $\xi$).
For $n=1$ we just have $\lambda_{1}(\xi)=\omega(\underline{\xi})=\omega(\xi)$, and similarly for the other
exponents.
The exponents $w_{n}(\xi)$ date back to Mahler~\cite{mahler}, 
others have been defined in~\cite{bula}. For $n=1$ we have $\widehat{\lambda}_{1}(\xi)=1$ for any irrational $\xi$,
as already noticed in Section~\ref{sek1}. The consequence $\widehat{\omega}(\underline{\xi})\leq 1$ for any $\underline{\xi}\notin \mathbb{Q}^{n}$ agrees with \eqref{eq:mar}.
Without restriction to the Veronese curve the uniform exponent $\widehat{\omega}(\underline{\xi})$ attains the maximum value $1$
for certain $\underline{\xi}\in\mathbb{R}^{n}$ 
with $\mathbb{Q}$-linearly independent coordinates with $\{1\}$
no matter how large $n$ is (see Poels~\cite{poels} for general
classes of manifolds containing points with this property, and further references). On the other hand,
on the Veronese curve the exponent turns out to be
always significantly smaller.
Only for $n=2$ the optimal bound is known, given as
\begin{equation} \label{eq:n2}
\widehat{\lambda}_{2}(\xi)\leq \frac{\sqrt{5}-1}{2}=0.6180\ldots.
\end{equation} 
The inequality was found by Davenport and Schmidt~\cite{davsh},
the optimality is due to Roy~\cite{royjl}. 
For even $n$ the following estimates are from~\cite{equprin}.

\begin{theorem}[Schleischitz] \label{schlei}
	Let $n\geq 2$ be an even integer. For any 
	transcendental real number $\xi$
	we have $\widehat{\lambda}_{n}(\xi)\leq \tau_{n}$ where 
	 $\tau_{n}$ is the solution of
	\[
	(\frac{n}{2})^{n}t^{n+1}- (\frac{n}{2}+1)t+1=0
	\]
	in the interval $(\frac{2}{n+2},\frac{2}{n})$.
\end{theorem}

The bound $\tau_{n}$ is obtained as the solution for
identity in \eqref{eq:mar} for $\widehat{\lambda}_{n}(\xi)$ 
when $\lambda_{n}(\xi)=2/n$.
For $n=2$ we confirm the optimal estimate \eqref{eq:n2}.
Other numerical bounds are
\begin{equation}  \label{eq:comp}
\widehat{\lambda}_{4}(\xi)\leq 0.370635\ldots, \quad
\widehat{\lambda}_{6}(\xi)\leq 0.268186\ldots, \quad
\widehat{\lambda}_{20}(\xi)\leq 0.092803\ldots.
\end{equation}
For large $n$,
the bound is of order
\begin{equation} \label{eq:asy}
\tau_{n}= \frac{2}{n}-\frac{\chi}{n^{2}}+o(n^{-2}), \qquad n\to\infty,
\end{equation}
where $\chi=3.18\ldots$ can be expressed as a zero of some
power series, see~\cite{equprin}. For sake of completeness,
we remark that for the case of odd $n$
the bound $\widehat{\lambda}_{n}(\xi)\leq 2/(n+1)$ was established by Laurent~\cite{laurent}, an improvement for $n=3$ is due to Roy~\cite{roy3}.
Application of \eqref{eq:01} from our new Theorem~\ref{new} leads to a small improvement that can be stated in the following way.

\begin{theorem}  \label{s2}
	Let $n\geq 4$ be an even integer
	and $\xi$ be any transcendental real number. 
	For $\alpha\in [1/n,1]$ define
	\[
	\epsilon_{\alpha}=1-\alpha-\frac{n}{2}\alpha^{2}-\cdots-\left(\frac{n}{2}\right)^{n-1}\alpha^{n}, \qquad\quad S_{\alpha}= \sum_{j=1}^{n} \left(\frac{n\alpha}{2}+\frac{2^{n+1}\epsilon_{\alpha}}{n^{n-1}\alpha^{n}}\right)^{1-j}.
	\]
	Then $\widehat{\lambda}_{n}(\xi)\leq \sigma_{n}$ with $\sigma_{n}$
	the solution of the implicit equation
	\begin{equation} \label{eq:imp}
	\frac{(\frac{2}{n}-\frac{16\epsilon_{\alpha}}{n^{2}\alpha^{2}})
		S_{\alpha}}{{(\frac{n\alpha}{2}-\frac{2^{n+1}\epsilon_{\alpha}}{n^{n}\alpha^{n}})^{-n}+(\frac{2}{n}-\frac{16\epsilon_{\alpha}}{n^{2}\alpha^{2}})(1-S_{\alpha})}}= \mu_{n}:= 
	\max\{ 2n-2, w(n)\},
	\end{equation}
	for $\alpha$ in the interval $(0,1)$, where $w(n)$ is the solution
	of 
	\[
	\frac{(n-1)w}{w-n}-w+1=  \left(\frac{n-1}{w-n}\right)^{n}
	\]
	in the interval $[n,2n-1)$. Precisely for $n\geq 10$ we get $\mu_{n}=2n-2$.
\end{theorem}

The left hand side in \eqref{eq:imp} is \eqref{eq:01} with $\beta=2/n$.
It can be checked and follows from the proof that 
$\sigma_{n}<\tau_{n}$ for $n\geq 4$. However, the quantities differ 
only by a very small amount, for example we obtain the numerical bounds
\[
\widehat{\lambda}_{4}(\xi)\leq 0.370629\ldots, \qquad
\widehat{\lambda}_{6}(\xi)\leq 0.268183\ldots,
\]
that may be compared with \eqref{eq:comp}. The asymptotics
\eqref{eq:asy} remain unaffected and the improvement occurs
in a lower order term. Improvements can be made via improving
the estimate \eqref{eq:01} in Theorem~\ref{new}. Another source of
improvement would be better bounds for the exponent $\widehat{w}_{n}(\xi)$,
related to $\mu_{n}$ in the theorem.
On the other hand,
increasing the bound for $\epsilon$ in \eqref{eq:epsilon}
would not lead to stronger bounds when combined with Theorem~\ref{new}
in its present form.

The sensitive dependence of \eqref{eq:01}, \eqref{eq:03}, \eqref{eq:wobene}, \eqref{eq:wobene2}
on $\epsilon=\epsilon_{\alpha,\beta}$ is the key problem why 
the improvement compared to Theorem~\ref{schlei} is small. 
We wonder about estimates in the optimal case $\epsilon=0$.
This is partly motivated by the fact that for $n=2$ identity 
in \eqref{eq:n2} is attained (only) for numbers $\xi$ where
$(\xi,\xi^{2})$ induces a regular graph $\epsilon=0$, see~\cite{royjl}.
We also include a result for the case where the difference
from the regular graph is very small, derived by continuity reasoning.
Mind the difference between $\epsilon$ and $\varepsilon$ therein.

\begin{theorem} \label{regg}
	Let $n\geq 4$ be even and let $\xi\in\mathbb{R}$. 
	First assume $\underline{\xi}=(\xi,\xi^{2},\ldots,\xi^{n})$ induces equality in \eqref{eq:mar}.
	If $n\in\{4,6,8\}$, then 
	\begin{equation}  \label{eq:b46}
	\widehat{\lambda}_{4}(\xi)< 0.3588, \qquad \widehat{\lambda}_{6}(\xi)< 0.2540,\qquad \widehat{\lambda}_{8}(\xi)< 0.1968,
	\end{equation}
	and as $n\to\infty$ we have the asymptotical bound 
	\begin{equation}  \label{eq:asyb}
	\widehat{\lambda}_{n}(\xi)\leq \frac{\Theta+o(1)}{n}, \qquad n\to\infty,
	\end{equation}
	where $\Theta=1.7564\ldots$ is the solution to $e^{t}/t=2\sqrt{e}$
	with $t>1$.
	
	When we drop equality assumption in \eqref{eq:mar}, 
	we still have the following continuity result.
	For $n\geq n_{0}$ and every $\varepsilon>0$, there exists
	$\delta_{n}>0$ such that if $\alpha=\widehat{\lambda}_{n}(\xi)\in[1/n,1]$ and
	$\beta=\lambda_{n}(\xi)\in[1/n,\infty]$ 
	are linked by $\epsilon=\epsilon_{\alpha,\beta}<\delta_{n}$, then
	we have
	\begin{equation}  \label{eq:conti}
	\widehat{\lambda}_{n}(\xi) \leq \frac{\Theta+\varepsilon}{n}.
	\end{equation}
\end{theorem}

Unfortunately, similar to Corollary~\ref{dadcor} the bound $\delta_{n}$ 
for the conclusion
\eqref{eq:conti} is very small when $\varepsilon$ is small. 
On the other hand, we remark that for $n=8$ the bound 
in \eqref{eq:b46} is already smaller than the left interval end point
$2/(n+2)=0.2$ from Theorem~\ref{schlei}. 

We derive a corollary on approximation to real numbers
by algebraic integers.
For an algebraic integer $\alpha$ denote by $H(\alpha)$
its height, i.e. the naive height (maximum modulus among
its coefficients) of its irreducible minimal
polynomial over $\mathbb{Z}[T]$. By a well-known argument that
originates in work
of Davenport and Schmidt~\cite{davsh}, from Theorem~\ref{s2} 
and Theorem~\ref{regg} we infer
\begin{corollary}
	Let $n\geq 4$ be an even integer, $\xi$ a real transcendental
	number
	and $\varepsilon>0$. Then for $\sigma_{n}$ as in Theorem~\ref{s2}
	the inequality
	\begin{equation} \label{eq:bestk}
	|\xi-\alpha| < H(\alpha)^{-\frac{1}{\sigma_{n}}-1+\varepsilon}
	\end{equation}
	has infinitely many solutions in real algebraic integers $\alpha$
	of degree at most $n+1$ (or real algebraic numbers of degree precisely $n$).
	If $\underline{\xi}=(\xi,\xi^{2},\ldots,\xi^{n})$ satisfies equality in \eqref{eq:mar}, then for $n\geq n_{0}$ and $\Theta$ as in Theorem~\ref{regg} the estimate
	\begin{equation}  \label{eq:ohnmacht}
	|\xi-\alpha| < H(\alpha)^{-\frac{n}{\Theta}-1+\varepsilon}
	\end{equation}
	has infinitely many solutions $\alpha$ as above.
\end{corollary}

The deduction for algebraic integers of degree $n+1$
is immediate
from~\cite[Lemma~1]{davsh}.
For the implication on algebraic numbers of precise degree $n$,
see Bugeaud and Teuli\'{e}~\cite{buteu}. 
An unconditional proof of \eqref{eq:ohnmacht} would be highly desirable, which would improve on 
the currently best known bounds of order $-n/2-O(1)$ from~\cite{buteu, davsh}
for the exponent in \eqref{eq:bestk} regarding
both approximation
by algebraic integers
of degree $n+1$ and algebraic numbers of exact degree $n$. 
However, we remark that concerning approximation by algebraic numbers
of degree at most $n$, a slightly stronger bound with exponent smaller than
$-n/\sqrt{3}-1=-n/1.73\ldots-1$ 
has recently been settled in~\cite{bads}.

\section{Proof of Theorem~\ref{new}}  \label{prof}

We first introduce parametric geometry of numbers that our proof is based on.
We follow the introductory paper of 
Schmidt and Summerer~\cite{ss}. Let $n\geq 1$ an integer and $\underline{\xi}=(\xi_{1},\ldots,\xi_{n})\in\mathbb{R}^{n}$ be fixed. For a parameter $Q>1$ and every $1\leq j\leq n+1$, let
$\psi_{j}(Q)$ be the infimum 
of $\nu$ for which the system
\[
|x|\leq Q^{1+\nu}, \qquad Y_{\underline{x}} \leq Q^{-1/n+\nu}
\]
has $j$ linearly independent integral solution vectors $\underline{x}=(x,y_{1},\ldots,y_{n})$, with $Y_{\underline{x}}$
as in \eqref{eq:rihs}.
Let $q=\log Q$ and define the functions $L_{j}(q)$ from $\psi_{j}(Q)$ via
\[
\psi_{j}(Q)= \frac{L_{j}(q)}{q}, \qquad 1\leq j\leq n+1.
\]
It can be checked that $L_{j}(q)$ are piecewise linear with slopes
among $\{-1,1/n\}$. In fact, any $L_{j}(q)$ is locally induced
by a function 
\begin{equation}  \label{eq:function}
L_{\underline{x}}(q)= \max \{ \log x-q, \; \log Y_{\underline{x}}+\frac{q}{n} \},
\end{equation}
more precisely we may write
\begin{equation}  \label{eq:japj}
L_{j}(q)= \min \; \max_{1\leq i\leq j} L_{\underline{x}_{i}}(q),
\end{equation}
with minimum taken over all sets of $j$
linearly independent vectors $\underline{x}_{1},\ldots,\underline{x}_{j}\in\mathbb{Z}^{n+1}$.
We infer
\[
-1\leq \underline{\psi}_{j} \leq \overline{\psi}_{j}\leq \frac{1}{n}, \qquad\qquad 1\leq j\leq n+1,
\]
where we have put
\[
\underline{\psi}_{j}= \liminf_{Q\to\infty} \psi_{j}(Q)=\liminf_{q\to\infty} \frac{L_{j}(q)}{q} ,\qquad \overline{\psi}_{j}= \limsup_{Q\to\infty} \psi_{j}(Q)=\limsup_{q\to\infty} \frac{L_{j}(q)}{q}.
\]
Another important property of the functions highlighted in~\cite{ss} is
\begin{equation} \label{eq:sump}
\left|\sum_{j=1}^{n+1} L_{j}(q)\right|= O(1).
\end{equation}
This reflects Minkowski's Second Convex Body Theorem.
In particular, on long intervals on average one function $L_{j}(q)$ will decay
with slope $-1$ whereas the remaining functions rise with slope $1/n$.
The next result linking the quantities $\underline{\psi}_{n+1}, \overline{\psi}_{n+1}$ with classical exponents
originates in German's paper~\cite{og}, see also Schmidt and Summerer~\cite{ss}.

\begin{proposition} \label{pro1}
	We have
	\begin{equation}  \label{eq:umrechnen}
	(1+\widehat{\omega}^{\ast}(\underline{\xi})^{-1})(1+\underline{\psi}_{n+1})= \frac{n+1}{n}.
	\end{equation}
	and
	\begin{equation} \label{eq:02}
	(1+\omega^{\ast}(\underline{\xi})^{-1})(1+\overline{\psi}_{n+1})= \frac{n+1}{n}.
	\end{equation}
\end{proposition}

\begin{proof}
	Define $\omega_{n+1}(\underline{\xi})$ and $\widehat{\omega}_{n+1}(\underline{\xi})$ as the supremum
	of $\omega$ such that
	\eqref{eq:bru}
	has $n+1$ linearly independent solution vectors
	$(x,y_{1},\ldots,y_{n})$ for arbitrarily large $X$ and
	all large $X$, respectively. 
	An application of Mahler's Theorem on Polar Convex Bodies implies 
	$\widehat{\omega}^{\ast}(\underline{\xi})^{-1}= \omega_{n+1}(\underline{\xi})$ and similarly $\omega^{\ast}(\underline{\xi})^{-1}= \widehat{\omega}_{n+1}(\underline{\xi})$, as observed in~\cite[Corollary~8.5]{og}.
	With these identifications, the identities just become
	a special case of~\cite[Proposition~7.1]{og} (see also~\cite[Theorem~1.4]{ss} for a special case whose 
	proof can be readily extended).
\end{proof}

The following direct consequence of Theorem~\ref{v}
is just formulated for convenience.

\begin{proposition}
	Let $\underline{\xi}$ satisfy the hypothesis of Theorem~\ref{v}
	and consider the induced sequence $(\underline{x}_{j})_{j\geq 1}$ with $\underline{x}_{j}=(x_{j},y_{j,1},\ldots,y_{j,n})$ from
	its claim.
	Then for $\ell\geq 0$ any fixed integer as $j\to\infty$ we have
	\begin{equation} \label{eq:wesee}
	\left(\left(\frac{\alpha}{\beta}-\phi\right)^{\ell}+o(1)\right)\log x_{j+\ell}\; \leq \; 
	\log x_{j}\; \leq \; \left(\left(\frac{\alpha}{\beta}+\phi\right)^{\ell}+o(1)\right)\log x_{j+\ell}, \quad j\geq 1,
	\end{equation}
	where
	\[
	\phi:= \frac{4\epsilon \beta^{n-1}}{\alpha^{n}}.
	\]
	Moreover, again as $j\to\infty$ we have
	\begin{equation}  \label{eq:nono}
	(-\beta-\rho+o(1))\log x_{j}\; \leq \; \log Y_{j}\leq (-\beta+\rho+o(1))\log x_{j}, \qquad\qquad j\geq 1,
	\end{equation}
	where
	\[
	\rho:= \frac{4\epsilon \beta^{2}}{\alpha^{2}}.
	\]
\end{proposition}

\begin{proof}
	The first property of Theorem~\ref{v} can be formulated
	\[
	\log x_{j}= (\frac{\alpha}{\beta}+\phi_{j}+o(1))\log x_{j+1}, \qquad j\to\infty,
	\]
	with some $|\phi_{j}|\leq \phi$.
	The first claim \eqref{eq:wesee} follows. The second
	claim \eqref{eq:nono} follows directly from the second property 
	of Theorem~\ref{v}.
	\end{proof}

In the proof, in accordance with notation in \eqref{eq:rihs} we write
\[
Y_{j}= \max_{1\leq i\leq n} |x_{j}\xi_{i}-y_{j,i}|, \qquad\qquad j\geq 1,
\]
when $\underline{x}_{j}=(x_{j},y_{j,1},\ldots,y_{j,n})$ are the vectors
from Theorem~\ref{v}.

\begin{proof}[Proof of Theorem~\ref{new}]
	We start with \eqref{eq:01}, which turns out to be the 
	most tedious calculation. In view of \eqref{eq:umrechnen},
	we have to provide lower bounds for $\underline{\psi}_{n+1}$.
	For $k\geq 1$, let $q_{k}$ denote the minimum of $L_{\underline{x}_{k}}$.
	It obviously follows from Theorem~\ref{v} that $q_{k}<q_{k+1}$ for 
	all large $k$.
	By the last property of Theorem~\ref{v} each $q_k$ is
	a local minimum of $L_{1}$, however they may not capture all such minima.
	
	We estimate the first $n$ functions
	$L_{1}(q), L_{2}(q),\ldots,L_{n}(q)$ from above
	in intervals $q\in I_{k}:=[q_{k},q_{k+1})$. Since $[q_{m},\infty)$ is the disjoint
	union of these intervals $I_{k}$ over $k\geq m$ and
	by \eqref{eq:sump}
	\begin{equation} \label{eq:finally}
	\psi_{n+1}(e^{q})=\frac{L_{n+1}(q)}{q}\geq  -\sum_{j=1}^{n}\frac{L_{j}(q)}{q}-O(q^{-1}), \qquad q>0,
	\end{equation}
	this will lead to the desired lower bound for $\underline{\psi}_{n+1}$.
	
	Notice now that the sets of vectors
	\[
	\{ \underline{x}_{k-n+1}, \underline{x}_{k-n+2}, \ldots, \underline{x}_{k} \} ,\qquad \{ \underline{x}_{k-n+2}, \underline{x}_{k-n+3}, \ldots, \underline{x}_{k+1}  \}
	\]
	are both linearly independent by the third claim of Theorem~\ref{v}.
	By \eqref{eq:japj} we infer that for any $q\in[q_{k},q_{k+1})$ 
	every function $L_{j}(q)$ for $1\leq j\leq n$ 
	is bounded above by $L_{\underline{x}_{g}}(q)$ 
	for some $k-n+1\leq g\leq k+1$ ($g$ depends on $j$ and $q$). 
	Concretely, we have
	\begin{equation}  \label{eq:ex}
	\sum_{j=1}^{n}\frac{ L_{j}(q)}{q}\leq \min\left\{ \sum_{j=1}^{n}\frac{ L_{\underline{x}_{k-n+j}}(q)}{q}, \sum_{j=2}^{n+1} \frac{L_{\underline{x}_{k-n+j}}(q)}{q}\right\},
	\qquad q\in [q_{k},q_{k+1}).
	\end{equation}
	We have to estimate the right hand side in dependence of $q$.
	Let $r_{k}$ be the intersection point of 
	$L_{\underline{x}_{k}}$ and $L_{\underline{x}_{k+1}}$.
	The values $r_{k}$ are closely connected to the
	local maxima of $L_{1}$, but do not necessarily 
	coincide with them. We have $q_{k}<r_{k}<q_{k+1}$.
		Keep in mind that by assumption $L_{\underline{x}_{k}}(r_{k})= L_{\underline{x}_{k+1}}(r_{k})$.
	First, we observe that we can restrict to intervals $[q_{k},r_{k}]$.
	Indeed, the function $L_{\underline{x}_{k+1}}$ decreases with slope $-1$
	in $(r_{k},q_{k+1}]$, and since the others have slope at most $1/n$
	the sum of all slopes is at most $-1+(n-1)\cdot n^{-1}=-1/n<0$ 
	within $(r_{k},q_{k+1}]$. Since 
	the expression in \eqref{eq:ex} cannot exceed $-1/n+o(1)$ 
	as $q\to\infty$ by \eqref{eq:sump} and $L_{n+1}(q)\leq q/n$,
	thus the sum decays in $(r_{k},q_{k+1}]$. Hence
	the worst (largest) bound in \eqref{eq:ex} and thus the worst (smallest) bound in \eqref{eq:finally} is 
	attained at the left endpoint $r_{k}$. 
	Thus indeed we can restrict to $[q_{k},r_{k}]$. 
	
	Next we further split the interval $[q_{k},r_{k}]$ into two
	intervals $[q_{k},s_{k})$ and $[s_{k},r_{k}]$ where $s_{k}$
	is the first coordinate of the point where $L_{\underline{x}_{k-n+1}}(q)$
	meets $L_{\underline{x}_{k+1}}(q)$, by \eqref{eq:function}
	that is the solution for $q$ of
	\[
	\log x_{k+1}-q =\log Y_{k-n+1}+\frac{q}{n},
	\]
	which yields
	\begin{equation} \label{eq:sk}
	s_{k}= \frac{n}{n+1}(\log x_{k+1}-\log Y_{k-n+1}).
	\end{equation}
	In the interval $q\in [q_{k}, s_{k}]$ we estimate 
	the left expression in \eqref{eq:ex}, that is
	\begin{equation}  \label{eq:norto}
	\sum_{j=1}^{n}\frac{L_{\underline{x}_{k-n+j}}(q)}{q}.
	\end{equation}
	Since every $L_{\underline{x}_{i}}(q)$ reaches its minimum at
	$q_{i}$ and the sequence $(q_{i})_{i\geq 1}$ is clearly increasing, 
	in this interval all
	involved functions increase with slope $1/n$. Thus
	\[
	L_{\underline{x}_{i}}(q_{k})= \log Y_{i}+\frac{q_{k}}{n}, \qquad 1\leq i\leq k,
	\]
	and the sum \eqref{eq:norto}
	reaches its maximum at the right end point $q=s_{k}$
	and we conclude
	\begin{align*}
	\max_{q\in [q_{k},s_{k}]} \sum_{j=1}^{n}\frac{ L_{\underline{x}_{k-n+j}}(q)}{q}&\leq
	\frac{\sum_{j=1}^{n} L_{\underline{x}_{k-n+j}}(s_{k})}{s_{k}}\\
	&\leq
	\frac{\sum_{j=1}^{n} (\log Y_{k-n+j}+\frac{s_{k}}{n})}{s_{k}}= 1+\frac{\sum_{j=1}^{n}\log Y_{k-n+j}}{s_{k}}.
	\end{align*}
	Thus by \eqref{eq:sk} we infer
	\begin{equation}  \label{eq:leftie}
	\max_{q\in [q_{k},s_{k}]} \sum_{j=1}^{n}\frac{ L_{\underline{x}_{k-n+j}}(q)}{q}\leq
	1+\frac{n+1}{n}\cdot \frac{\sum_{j=1}^{n}\log Y_{k-n+j}}{\log x_{k+1}-\log Y_{k-n+1}}.
	\end{equation}
	Finally, in the interval $q\in [s_{k}, r_{k}]$ we estimate 
	the right expression in \eqref{eq:ex}, that is
	\begin{equation}  \label{eq:thesum}
	\sum_{j=2}^{n+1} \frac{L_{\underline{x}_{k-n+j}}(q)}{q}.
	\end{equation}
	Since in this interval 
	$L_{\underline{x}_{k+1}}(q)$ decays with slope $-1$,  again
	the sum of slopes in this interval is at most $-1+(n-1)\cdot n^{-1}=-1/n<0$. Thus the most disadvantageous case of
	a maximum in \eqref{eq:ex} is attained
	at the left endpoint $s_{k}$,  which leads to the same bound 
	as in \eqref{eq:leftie} again.
	
	This bound \eqref{eq:leftie} remains to be estimated, which 
	we perform via Theorem~\ref{v}.
	We first readily verify that by $x_{k+1}>1$ and $Y_{i}<1$ the expression is
	increasing in all involved variables $x_{k+1},Y_{k-n+1},\ldots,Y_{k}$. Thus we have to find upper
	bounds for each variable. 
	
	By \eqref{eq:ex} and \eqref{eq:nono} applied to $j$ from $j=k-n+1$ up to $j=k$, we obtain
	\begin{equation}  \label{eq:thistime}
	\max_{q\in[q_{k},q_{k+1}]} \sum_{j=1}^{n} \frac{L_{j}(q)}{q} \leq 
	1+\frac{n+1}{n}\cdot \frac{\sum_{j=1}^{n} (-\beta+\rho+o(1))\log x_{k-n+j}}{\log x_{k+1}-(-\beta+\rho+o(1))\log x_{k-n+1}}, \qquad k\to\infty.
	\end{equation}
	Observe that $\rho<\beta$ as this is equivalent to $\alpha^{2}/\beta>4\epsilon$, but
	since $n\geq 2$ and $0<\alpha<\beta$ we compute
	\begin{equation}  \label{eq:klaro}
	\frac{\alpha^{2}}{\beta}=\beta\left(\frac{\alpha}{\beta}\right)^{2}\geq \beta \left(\frac{\alpha}{\beta}\right)^{n}\geq  (\beta-\alpha)\left(\frac{\alpha}{\beta}\right)^{n}>
	\frac{\left(\frac{\alpha}{\beta}\right)^{n}(\beta-\alpha)}{n}\geq 4\epsilon,
	\end{equation}
	where the most right inequality follows from \eqref{eq:epsilon}.
	Hence 
	we verify that this time the expression \eqref{eq:thistime}
	is decreasing in $x_{k-n+2},\ldots,x_{k}$ but increasing in $x_{k+1}$.
	For $x_{k-n+1}$ the situation is unclear, depending on the sign
	of $\log x_{k+1}-(\beta-\rho)\sum_{j=2}^{n} \log x_{k-n+j}$. 
	Thus we want to find lower bounds for $x_{k-n+2},\ldots,x_{k}$
	and upper bounds for $x_{k+1}$ in terms of $x_{k-n+1}$.
	
	It can be checked similar to \eqref{eq:klaro} that $\alpha/\beta-\phi>0$.
	By \eqref{eq:wesee} we can estimate
	\[
	\log x_{k-n+j}\geq \frac{\log x_{k-n+1}}{(\frac{\alpha}{\beta}+\phi)^{j-1}+o(1)},\; 1\leq j\leq n,\qquad 
	\log x_{k+1}\leq \frac{\log x_{k-n+1}}{(\frac{\alpha}{\beta}-\phi)^{n}+o(1)}.
	\]
	We remark that these are crude estimates that can be improved
	by a refined analysis of the interplay of 
	the possible quotients of the occurring $\log x_{j}$, however 
	calculations become rather cumbersome. Similar situations will occur more
	often below when we apply Proposition~\ref{pro1}.
	Inserting these bounds in \eqref{eq:thistime} we divide numerator and denominator by $x_{n-k+1}$ to conclude
	\begin{align*}
	\max_{q\in[q_{k},q_{k+1}]}\sum_{j=1}^{n} \frac{L_{j}(q)}{q} &\leq 
	1+\frac{n+1}{n}\cdot \frac{\sum_{j=1}^{n} (-\beta+\rho+o(1))\frac{1}{(\frac{\alpha}{\beta}+\phi)^{j-1}+o(1)}}{\frac{1}{(\frac{\alpha}{\beta}-\phi)^{n}+o(1)}+\beta-\rho-o(1)}\\&=
	1+\frac{n+1}{n}(-\beta+\rho)\cdot \frac{\sum_{j=1}^{n} \frac{1}{(\frac{\alpha}{\beta}+\phi)^{j-1}}}{\frac{1}{(\frac{\alpha}{\beta}-\phi)^{n}}+\beta-\rho}+o(1), \qquad k\to\infty.
	\end{align*}
	Now we finally use \eqref{eq:finally} to conclude
	\[
	\underline{\psi}_{n+1}= \liminf_{q\to\infty} \frac{L_{n+1}(q)}{q}
	= \liminf_{k\to\infty} \min_{q\in[q_{k},q_{k+1})} \frac{L_{n+1}(q)}{q}
	\geq  -1+\frac{n+1}{n}(\beta-\rho)\cdot \frac{\sum_{j=1}^{n} \frac{1}{(\frac{\alpha}{\beta}+\phi)^{j-1}}}{\frac{1}{(\frac{\alpha}{\beta}-\phi)^{n}}+\beta-\rho}.
	\]
	Using \eqref{eq:umrechnen} leads to the stated bound 
	\eqref{eq:01} after some rearrangements.
	
	We turn to the lower bound \eqref{eq:03} for $\omega^{\ast}(\underline{\xi})$.
	Here we look at the values $q=q_{k}$ where the functions
	$L_{\underline{x}_{k}}(q)$ 
	attain their minimum values. Again since $(q_{i})_{i\geq 1}$ increases,
	all
	\[
	L_{\underline{x}_{k-n+1}}(q),\; L_{\underline{x}_{k-n+2}}(q),\;\ldots,\;
	L_{\underline{x}_{k-1}}(q)
	\]
	increase with slope $1/n$ at $q=q_{k}$.
	Hence
	\[
	L_{\underline{x}_{k-n+j}}(q_{k})= \log Y_{k-n+j}+\frac{q_{k}}{n}, \qquad 1\leq j\leq n,
	\]
	where for $j=n$ we can use that representation as well as 
	there is equality in the expressions of \eqref{eq:function} for
	$L_{\underline{x}_{k}}(q)$ at $q=q_{k}$.
	From the linear dependence of $L_{\underline{x}_{k-n+1}}, L_{\underline{x}_{k-n+2}},\ldots,
	L_{\underline{x}_{k-1}}$ by Theorem~\ref{v} 
	and \eqref{eq:finally}, if we
	let $Q_{k}=e^{q_{k}}$ we infer
	\[
	\psi_{n+1}(Q_{k})=\frac{L_{n+1}(q_{k})}{q_{k}}\geq -\sum_{j=1}^{n} \frac{L_{\underline{x}_{k-n+j}}(q_{k})}{q_{k}}-O(q_{k}^{-1})= -\frac{\sum_{j=1}^{n} \log Y_{k-n+j}}{q_{k}}-1-O(q_{k}^{-1}). 
	\]
	Now $q_{k}$ is by \eqref{eq:function} given as the solution of
	\[
	\log x_{k}-q_{k}= \log Y_{k}+\frac{q_{k}}{n}
	\]
	hence
	\[
	q_{k}= \frac{n}{n+1} \cdot (\log x_{k}-\log Y_{k}).
	\]
	Plugging in yields
	\[
	\psi_{n+1}(Q_{k})\geq \frac{n+1}{n}\cdot 
	\frac{\sum_{j=1}^{n} \log Y_{k-n+j}}{\log Y_{k}-\log x_{k}}-1-O(q_{k}^{-1}).
	\]
	We see that the right hand side expression is decreasing in $x_{k}$ and
	in $Y_{k-n+1},\ldots,Y_{k-1}$, the situation is unclear for $Y_{k}$
	depending on the sign of $\log x_{k}+\sum_{j=1}^{n-1}\log Y_{j}$.
	Thus we look for lower bounds for the other variables in terms of $Y_{k}$. 
	By \eqref{eq:nono} we have
	\[
	\psi_{n+1}(Q_{k})\geq \frac{n+1}{n}\cdot 
	\frac{\log Y_{k}-(\beta+\rho+o(1))\sum_{j=1}^{n-1} \log x_{k-n+j}}{\log Y_{k}(1+\frac{1}{\beta+\rho})+o(1)}-1, \qquad k\to\infty.
	\]
	Now this expression is increasing in $Y_{k}$ so
	again by \eqref{eq:nono} we can estimate
	\[
	\psi_{n+1}(Q_{k})\geq \frac{n+1}{n}\cdot 
	\frac{(\rho-\beta+o(1))\log x_{k}-(\beta+\rho+o(1))\sum_{j=1}^{n-1} \log x_{k-n+j}}{ (\rho-\beta+o(1))(1+\frac{1}{\beta+\rho})\log x_{k}}-1.
	\]
	Since $\rho-\beta<0$, the right hand side is increasing in $x_{k-n+1},\ldots,x_{k-1}$.
	Thus by \eqref{eq:wesee} we can replace the sum 
	in the numerator expression by $T+o(1)$ with
	\[
	T:= \sum_{j=1}^{n-1}(\frac{\alpha}{\beta}+\phi)^{j},
	\]
	without making the expression larger. Dividing denominator and numerator
	by $\log x_{k}$ and dropping the negligible lower order terms 
	yields
	\[
	\overline{\psi}_{n+1}\geq \limsup_{k\to\infty} \psi_{n+1}(Q_{k})
	\geq \frac{n+1}{n}\cdot\frac{(\rho-\beta)-(\beta+\rho)T}{ (\rho-\beta)(1+\frac{1}{\beta+\rho})}-1.
	\]
	Inserting in \eqref{eq:02} gives \eqref{eq:03}.
	
	We turn to the upper bounds. 
	Here we have to find upper bounds for $\overline{\psi}_{n+1}$
	and $\underline{\psi}_{n+1}$, respectively.
	For the uniform exponent
	we choose $q=u_{k}$ as the points where the graphs of
	$L_{\underline{x}_{k-n}}(q)$ and $L_{\underline{x}_{k}}(q)$ meet,
	that is the solution of
	\[
	L_{\underline{x}_{k-n}}(q)= L_{\underline{x}_{k}}(q).
	\]
	Since at this position clearly
	$L_{\underline{x}_{k-n}}(q)$ increases
	whereas $L_{\underline{x}_{k}}(q)$ decreases, 
	from \eqref{eq:function} we obtain
	\begin{equation} \label{eq:repr}
	u_{k}= \frac{n}{n+1}(\log x_{k}-\log Y_{k-n}).
	\end{equation}
	Again since $\underline{x}_{k-n}, \underline{x}_{k-n+1},\ldots,
	\underline{x}_{k}$ are linearly independent, by \eqref{eq:japj} 
	we have to bound 
	each of the corresponding
	functions $L_{\underline{x}_{j}}$, $k-n\leq j\leq k$, 
	at position $u_{k}$. It is easy to see that 
	$L_{\underline{x}_{k}}(u_{k})=L_{\underline{x}_{k-n}}(u_{k})$ 
	is the maximum among the values, since the rising part of the graph of
	$L_{\underline{x}_{k-n}}(q)$ will intersect the falling part of each of
	$L_{\underline{x}_{k-n+1}}(q),
	\ldots,L_{\underline{x}_{k-1}}(q)$ before (i.e. at smaller $q$ values) it 
	eventually meets $L_{\underline{x}_{k}}(q)$ at $q=u_{k}$. 
	Moreover by \eqref{eq:function} the function value is given as 
	$L_{\underline{x}_{k}}(u_{k})=\log x_{k}-u_{k}$. 
	Hence we estimate
	\[
	\underline{\psi}_{n+1} \leq \liminf_{k\to\infty} \frac{L_{\underline{x}_{k}}(u_{k})}{u_{k}}\leq
	\liminf_{k\to\infty} \frac{\log x_{k}-u_{k}}{u_{k}}= \liminf_{k\to\infty} \frac{\log x_{k}}{u_{k}}-1.
	\]
	Using the representation \eqref{eq:repr} of $u_{k}$ yields
	\[
	\underline{\psi}_{n+1} \leq \liminf_{k\to\infty} \frac{n+1}{n}\cdot \frac{\log x_{k}}{\log x_{k}-\log Y_{k-n}}-1.
	\]
	We check that the expression increases in $Y_{k-n}$.
	Hence we can use \eqref{eq:nono} to estimate 
	\[
	\underline{\psi}_{n+1} \leq \liminf_{k\to\infty} \frac{n+1}{n}\cdot \frac{\log x_{k}}{\log x_{k}-(-\beta+\rho)\log x_{k-n}}-1 
	\]
	Again since $\beta>\rho$ we see the right hand side decays in $x_{k-n}$, application of \eqref{eq:wesee} with $\ell=n$ and yields
	\begin{equation}  \label{eq:direktl}
	\underline{\psi}_{n+1} \leq \frac{n+1}{n}\cdot \frac{1}{1+(\beta-\rho)(\frac{\alpha}{\beta}-\phi)^{n}}-1.
	\end{equation}
	Plugging this into \eqref{eq:umrechnen} we get \eqref{eq:wobene}.
	
	Finally we show \eqref{eq:wobene2}. 
	For this we estimate $\overline{\psi}_{n+1}$ from above
	to apply \eqref{eq:02}.
	Recall we defined $u_{k}$ as the value
	where $L_{\underline{x}_{k-n}}$ meets $L_{\underline{x}_{k}}$.
	Consider the interval $[u_{k},u_{k+1})$. We have seen above
	that the last successive minimum function at $u_{k}$ is at most 
	\[
	L_{n+1}(u_{k})\leq L_{\underline{x}_{k}}(u_{k})= L_{\underline{x}_{k-n}}(u_{k})= \frac{u_{k}}{n}+\log Y_{k-n}.
	\]
	Now let $p_{k}$ be the point where $L_{\underline{x}_{k-n}}$ meets $L_{\underline{x}_{k+1}}$. Clearly $p_{k}>u_{k}$. 
	Similarly as in \eqref{eq:repr} above we get
	\begin{equation} \label{eq:fromm}
	p_{k}= \frac{n}{n+1}(\log x_{k+1}-\log Y_{k-n}).
	\end{equation}
	Comparison with \eqref{eq:repr} for index $k+1$ gives 
	$p_{k}<u_{k+1}$ since $(Y_{j})_{j\geq 1}$ decreases.
	 We split $[u_{k},u_{k+1})$ into $[u_{k},p_{k})$ and $[p_{k},u_{k+1})$.
	 It is readily verified that for $q\in [u_{k},p_{k})$ we have
	 \[
	 L_{\underline{x}_{k-n}}(q)>L_{\underline{x}_{k-n+1}}(q)>\cdots>L_{\underline{x}_{k}}(q)
	 \]
	 and since these vectors are linearly independent, we have
	 \begin{equation}  \label{eq:jippy}
	 L_{n+1}(q)\leq L_{\underline{x}_{k-n}}(q), \qquad q\in [q_{k},p_{k}).
	 \end{equation}
	 In the other partial interval $q\in [p_{k},u_{k+1})$ we similarly 
	 see that
	 \[
	  L_{\underline{x}_{k+1}}(q)>L_{\underline{x}_{k-n+1}}(q)>
	  L_{\underline{x}_{k-n+2}}(q)>\cdots>L_{\underline{x}_{k}}(q)
	 \]
	 and again by the linear independence 
	 of the $n+1$ functions $L_{\underline{x}_{k-n+1}}, L_{\underline{x}_{k-n+2}},\ldots, L_{\underline{x}_{k+1}}$
	 by Theorem~\ref{v}
	 and the definition of $u_{k+1}$
	 we conclude
	 \[
	 L_{n+1}(q)\leq L_{\underline{x}_{k+1}}(q), \qquad q\in [p_{k},u_{k+1}).
	 \]
	 Since $L_{\underline{x}_{k+1}}$ which decays with
	 slope $-1$ in the latter interval $[p_{k},u_{k+1})$, 
	 the corresponding values $L_{\underline{x}_{k+1}}(q)/q$ decrease
	 in this interval. Thus
	 we only need to take into account $[u_{k},p_{k}]$ when looking for
	 upper bounds for $\overline{\psi}_{n+1}$.
	 Moreover, by \eqref{eq:jippy} and since
	 $L_{\underline{x}_{k-n}}(q)$ increases in this interval
	 with slope $1/n$, the quantity $L_{\underline{x}_{k-n}}(q)/q$	
	 increases in $[u_{k},p_{k}]$,
it suffices to consider the right end point $q=p_{k}$. 
From \eqref{eq:fromm} and
since $[u_{m},\infty)$ is the disjoint union of the intervals $[u_{k},u_{k+1})$ over $k\geq m$, observing the similarity between \eqref{eq:repr} and \eqref{eq:fromm},
a very similar argument as for \eqref{eq:direktl} above
 yields
\[
\overline{\psi}_{n+1} \leq \limsup_{k\to\infty} \frac{L_{\underline{x}_{k+1}}(p_{k})}{p_{k}}  \leq 
\frac{n+1}{n}\cdot \frac{1}{1+(\beta-\rho)(\frac{\alpha}{\beta}-\phi)^{n+1}}-1.
\]
Again using \eqref{eq:02} we obtain the bound \eqref{eq:wobene2}.
	\end{proof}

We deduce the corollary.

\begin{proof}[Proof of Corollary~\ref{dadcor}]
	We identify $\alpha=\widehat{\omega}(\underline{\xi})$ and
	$\beta=\omega(\underline{\xi})$.
	Assume equality in \eqref{eq:mar}. Then $\epsilon=\rho=\phi=0$ and the bounds of Theorem~\ref{new} simplify reasonably, in fact \eqref{eq:wobene}
	and \eqref{eq:wobene2} directly become $\beta^{n-1}/\alpha^{n}$ and $\beta^{n}/\alpha^{n+1}$,
	respectively. To identify
	the lower bounds \eqref{eq:01}, \eqref{eq:03}
	with these
	respective values, some rearrangements and explicit use
	of equality in \eqref{eq:mar} is required, we do not
	carry it out.
	We describe how to infer the second claim
	by a continuity argument.
	Let $\varepsilon>0$.
	First fix $\alpha\in[1/n,1)$. Let $\beta_{0}$ be the solution
	to equality in \eqref{eq:mar}.
	Then by continuity 
	of the bounds in Theorem~\ref{new} in $\alpha,\beta$, there is some
	$\tilde{\delta}>0$ that depends on $n,\varepsilon,\alpha$
	such that $\beta\in (\beta_{0}-\tilde{\delta},\beta_{0}+\tilde{\delta})$
	implies that the bound expressions differ from the respective values
	$\omega(\underline{\xi})^{n-1}/\widehat{\omega}(\underline{\xi})^{n}$
	and
	$\omega(\underline{\xi})^{n}/\widehat{\omega}(\underline{\xi})^{n+1}$
	of the case $\beta=\beta_{0}$ by less than $\varepsilon$. By implicit function theorem, for
	given $\epsilon$ we have that the solution 
	$\beta$ to \eqref{eq:epsilon} depends continuously on $\alpha$, 
	and for $\epsilon=0$ it becomes $\beta_{0}$.
	Hence the same claim holds if $\epsilon=\epsilon_{\alpha,\beta}$
	in \eqref{eq:eps2}
	is smaller than some modified $\delta>0$.
	Since we restrict $\alpha$ to the compact interval $[1,c]$
	and (a lower bound for) $\delta$ depends continously on $\alpha$, we
	may choose $\delta>0$ independently of $\alpha$. 
\end{proof}

\section{Proof of Theorems~\ref{s2},~\ref{regg}}

The proof of Theorem~\ref{s2} is based on a case distinction
$\lambda_{n}(\xi)> \frac{2}{n}$ and $\lambda_{n}(\xi)\leq \frac{2}{n}$.
The first case is dealt with by the following result from~\cite[Section~4]{equprin}. 

\begin{theorem}[Schleischitz]  \label{s1}
	Let $n\geq 2$ be an even integer. Assume $\xi$ is transcendental
	real and satisfies 
	\[
	\lambda_{n}(\xi)> \frac{2}{n}.
	\] 
	Then 
	\[
	\widehat{\lambda}_{n}(\xi)\leq \frac{2}{n+2}.
	\]
\end{theorem}  

It can be verified that the bound $2/(n+2)$ is smaller than $\sigma_{n}$ in
Theorem~\ref{s2} (thus clearly also smaller than $\tau_{n}$ in Theorem~\ref{schlei}).

In the latter case $\lambda_{n}(\xi)\leq \frac{2}{n}$ we use 
Theorem~\ref{new}. We point out that the slightly weaker bounds 
in our Theorem~\ref{schlei} obtained in~\cite{equprin}
followed from combining Theorem~\ref{s1}
with \eqref{eq:mar}, corresponding to $\epsilon=0$
in \eqref{eq:eps2}. 
The key observation for the improvement is that when $\beta=2/n$
and $\epsilon=0$,
then the bound for the exponent $\widehat{w}_{n}(\xi)$ we obtain
from \eqref{eq:01} is larger 
than settled upper bounds for this exponent rephrased in Theorem~\ref{davesh} below. Thus by continuity we expect that if $\beta$
is not too much smaller than $2/n$ and $\epsilon$ sufficiently
small, we still get a contradiction to \eqref{eq:01}.
Hereby we want to recall the identifications $\alpha=\widehat{\omega}(\underline{\xi})=\widehat{\lambda}_{n}(\xi)$,
$\beta=\omega(\underline{\xi})=\lambda_{n}(\xi)$ and $\widehat{w}_{n}(\xi)=\widehat{\omega}^{\ast}(\underline{\xi})$.

We turn to the upper bound for $\widehat{w}_{n}(\xi)$ indicated above.
Indeed, in contrast to general points in $\mathbb{R}^{n}$ where
$\widehat{\omega}^{\ast}(\underline{\xi})=\infty$ may occur (if $n\geq 2$), for points
on the Veronese curve the uniform exponent is bounded in terms of $n$.
 The following currently best known bound is a consequence of Bugeaud and Schleischitz~\cite{buschlei}
 when incorporating the linear form result (dual to \eqref{eq:mar} in some sense) from~\cite{mamo}, already 
 observed in~\cite{acta2018}. 

\begin{theorem}[Bugeaud, Schleischitz] \label{davesh}
	For any $n\geq 1$ and transcendental real $\xi$ we have $\widehat{w}_{n}(\xi) \leq \mu_{n}$ 
	where $\mu_{n}$ is as defined in Theorem~\ref{s2}.
\end{theorem}

For sake of completeness we mention
that the paper~\cite{buschlei} in turn
improved slightly on a bound by Davenport and Schmidt~\cite{davsh}.
All aforementioned proofs are based on some variation (see~\cite[Lemma~8]{davsh}) of the Liouville inequality that provides effective lower bounds 
for the distance between two distinct algebraic numbers in terms of their
degrees and heights.
Recall \eqref{eq:clex} for the following proof.

\begin{proof}[Proof of Theorem~\ref{s2}]
	We may assume $\lambda_{n}(\xi)\leq 2/n$ by Theorem~\ref{s1}.
	Observe that by \eqref{eq:mar} this implies an upper bound for
	$\widehat{\lambda}_{n}(\xi)$ that is just slightly larger than $\sigma_{n}$.
	Denote this inferred bound by $\Psi_{n}$ and write $I_{n}:=(\sigma_{n},\Psi_{n})$.
	
	Now assume on the contrary that $\widehat{\lambda}_{n}(\xi)>\sigma_{n}$
	for some $\xi$. 
	Then $\widehat{\lambda}_{n}(\xi)\in I_{n}$, 
	and again from \eqref{eq:mar}
	we obtain a lower bound for $\lambda_{n}(\xi)$ just slightly
	smaller than $2/n$. Denote the bound by $\Phi_{n}$ and the 
	resulting range for $\lambda_{n}(\xi)$ by $J_{n}:=(\Phi_{n},2/n]$.
	For given parameters $\alpha,\beta$
	write 
	\[
	W_{\alpha,\beta}:= \frac{(\beta-\frac{4\epsilon\beta^{2}}{\alpha^{2}})
		S}{{(\frac{\alpha}{\beta}-\frac{4\beta^{n-1}\epsilon}{\alpha^{n}})^{-n}+(\beta-\frac{4\epsilon\beta^{2}}{\alpha^{2}})(1-S)}}
	\]
	where $\epsilon=\epsilon_{\alpha,\beta}$ is defined in \eqref{eq:eps2},
	and $S=S_{\alpha,\beta}$ in \eqref{eq:01}.
	Then in particular $W_{\alpha}:= W_{\alpha,2/n}$ 
	denotes the left hand side in \eqref{eq:imp}.
	By construction $\epsilon_{\alpha}=\epsilon_{\alpha,2/n}$ and $S_{\alpha}=S_{\alpha,2/n}$
	and $\sigma_{n}$ is the solution for $\alpha$
	to equality $W_{\alpha}=\mu_{n}$, thus
	\begin{equation}  \label{eq:jj}
	W_{\sigma_{n},2/n}= \mu_{n}.
	\end{equation} 
	A brief calculation verifies that
	$\epsilon=\epsilon_{\alpha,\beta}$ from \eqref{eq:eps2}
	satisfies \eqref{eq:epsilon} when
	\begin{equation}  \label{eq:asin}
	\alpha\in I_{n},
	\qquad \beta\in J_{n}\cup (\frac{2}{n},\frac{2}{n}+\varepsilon)=(\Phi_{n},\frac{2}{n}+\varepsilon)=:K_{n}
	\end{equation}
	for some small $\varepsilon=\varepsilon(n)>0$ (independent of $\alpha$). 
	
	Next observe that by the strict inequality $\widehat{\lambda}_{n}(\xi)>\sigma_{n}$, 
	we have that the
	hypothesis \eqref{eq:yesc} of Theorem~\ref{v}
	holds for every pair $(\alpha,\beta)$ with $\alpha\in(\sigma_{n},\widehat{\lambda}_{n}(\xi))\subseteq I_{n}$ 
	and $\beta>2/n$, and suitable $a,b$.
	Thus, hypothesis \eqref{eq:yesc} holds in particular
	for $\alpha\in I_{n}$ and $\beta\in (2/n,2/n+\varepsilon)\subseteq K_{n}$. Hence we can apply Theorem~\ref{new} for any 
	 \begin{equation} \label{eq:asin2}
	 \alpha\in (\sigma_{n},\widehat{\lambda}_{n}(\xi))\subseteq I_{n}, \qquad \beta\in (\frac{2}{n},\frac{2}{n}+\varepsilon)\subseteq K_{n},
	 \end{equation}
	 and \eqref{eq:01} yields for any pair $\alpha,\beta$ as in
	\eqref{eq:asin2} the inequality
	\begin{equation}  \label{eq:j1}
	\widehat{w}_{n}(\xi)=\widehat{\omega}^{\ast}(\underline{\xi})\geq 
	W_{\alpha,\beta}.
	\end{equation}  
	A short calculation further shows
	that when $\beta\in K_{n}$ is fixed, 
	the expression $W_{\alpha,\beta}$
	increases as $\alpha$ increases in $I_{n}$. Thus
	\[
	W_{\alpha,\beta}> W_{\sigma_{n},\beta},
	\]
	with strict inequality because $\alpha>\sigma_{n}$ strictly.
	By continuity of $W_{\alpha,\beta}$
	in the second argument, for any fixed $\alpha>\sigma_{n}$  
	we still have 
	\begin{equation} \label{eq:j2}
	W_{\alpha,\beta}>W_{\sigma_{n},2/n}
	\end{equation}
	if $\beta$ is sufficiently close to $2/n$
	(alternatively one can start with $\mu_{n}+\varepsilon$ for
	arbitrarily small $\varepsilon>0$ in the right hand side of
	\eqref{eq:imp}, and use continuity to derive the contradiction below). 
	Thus for any pair $\alpha,\beta$ as in
	\eqref{eq:asin2}, combining \eqref{eq:jj}, \eqref{eq:j1}, \eqref{eq:j2} 
	upon making $\varepsilon$ smaller if necessary we conclude
	\[
	\widehat{w}_{n}(\xi)=\widehat{\omega}^{\ast}(\underline{\xi})\geq 
	W_{\alpha,\beta}>W_{\sigma_{n},2/n}=\mu_{n}.
	\]
	This contradicts Theorem~\ref{davesh}.
	Thus we cannot have 
	$\widehat{\lambda}_{n}(\xi)>\sigma_{n}$.
	\end{proof}

Finally we prove Theorem~\ref{regg} with a similar method.

\begin{proof}[Proof of Theorem~\ref{regg}]
	We proceed as in the proof of Theorem~\ref{s2}.
	Let us first assume $\epsilon=0$ for some
	$\alpha,\beta$, i.e. there is equality in \eqref{eq:mar}
	and we are in the situation of the regular graph.
	Then $\alpha=\widehat{\lambda}_{n}(\xi)$ and $\beta=\lambda_{n}(\xi)$ and
	by Corollary~\ref{dadcor}, we have identity
	\begin{equation}  \label{eq:111}
	\widehat{w}_{n}(\xi)=\widehat{\omega}^{\ast}(\underline{\xi})= \frac{\omega(\underline{\xi})^{n-1}}{\widehat{\omega}(\underline{\xi})^{n}}= \frac{\lambda_{n}(\xi)^{n-1}}{\widehat{\lambda}_{n}(\xi)^{n}}.
	\end{equation}
	It is easily checked that upon equality in \eqref{eq:mar}
	 the expressions are increasing as functions in
	$\alpha=\widehat{\lambda}_{n}(\xi)$. 
	By Theorem~\ref{davesh}, the exponent $\widehat{\lambda}_{n}(\xi)$
	is thus bounded by the solution to
	\begin{equation} \label{eq:mueh}
	\mu_{n}= \frac{\lambda_{n}(\xi)^{n-1}}{\widehat{\lambda}_{n}(\xi)^{n}},
	\end{equation}
	with 
	the exponents $\lambda_{n}(\xi)$ and $\widehat{\lambda}_{n}(\xi)$
	linked by an identity in \eqref{eq:mar}. 
		For $n\in\{4,6,8\}$
	we derive the stated numerical bounds \eqref{eq:b46} with some computation.
	For large $n$, we have $\mu_{n}=2n-2<2n$ and with some analysis
	of the regular graph the claimed asymptotics \eqref{eq:asyb} can be 
	derived. We give some more details.
	We use identity~\cite[(31)]{jp} which,
    upon identifying $\widehat{\omega}_{n+1}(\underline{\xi})^{-1}=\omega(\underline{\xi})$
    by~\cite[Corollary~8.5]{og} (already used in the proof of Proposition~\ref{pro1}), can be written
	\begin{equation}  \label{eq:lefths}
	(1+\omega^{\ast}(\underline{\xi}))\cdot \left(1+\frac{1}{\omega^{\ast}(\underline{\xi})}\right)^{n}= \left(1+\frac{1}{\omega(\underline{\xi})}\right)\cdot (1+\omega(\underline{\xi}))^{n}.
	\end{equation}
	Moreover by \eqref{eq:111} and \eqref{eq:mueh} we have $\widehat{\omega}^{\ast}(\underline{\xi})/n=\mu_{n}/n=2-o(1)$ and
	then also $\omega^{\ast}(\underline{\xi})/n= 2+o(1)$
	as can be derived from identity~\cite[(33)]{jp}, we see that
	the left hand side in \eqref{eq:lefths} is of order $(2\sqrt{e}+o(1))n$ as $n\to\infty$. Thus so is the right hand side
	and we readily conclude $n\omega(\underline{\xi})= \Theta+o(1)$.
	Finally the smaller quantity $n\widehat{\omega}(\underline{\xi})=n\widehat{\lambda}_{n}(\xi)$ 
	will be asymptotically of the same order
	as $n\to\infty$ (see (30) in~\cite{jp}), thus \eqref{eq:asyb} follows.
	
	Finally \eqref{eq:conti} follows by a similar continuity argument
	as in Corollary~\ref{dadcor}. First assume $\alpha\in[1/n,1)$ is fixed.
	By Corollary~\ref{dadcor} and its proof, the expression $\widehat{w}_{n}(\xi)=\widehat{\omega}^{\ast}(\underline{\xi})$
	depends continuously on $\epsilon$ if $\beta=\beta(\epsilon)$ is such that
	there is identity in \eqref{eq:epsilon}. 
	Moreover for $\epsilon=0$ and $\alpha=\widehat{\lambda}_{n}(\xi)$
	larger than claimed, we get a contradiction $\widehat{w}_{n}(\xi)>\mu_{n}$
	as we have proved above.
	Thus for $\epsilon\in [0,\delta_{n}]$
	with $\delta_{n}>0$ 
	small enough in dependence
	of $n,\alpha,\varepsilon$, in case of larger $\alpha=\widehat{\lambda}_{n}(\xi)$
	we still obtain the same contradiction 
	$\widehat{w}_{n}(\xi)>\mu_{n}$. Finally, again as we can restrict $\alpha=\widehat{\lambda}_{n}(\xi)$
	to a compact interval like $[1/n,1/2]$, 
	we can choose $\delta_{n}$ uniformly in $\alpha$,
	depending only on $n$ and $\varepsilon$. 
\end{proof}

\vspace{0.5cm}

{\em The author thanks the referee for the careful reading 
and for providing references, as well as
Nikolay Moshchevitin for pointing out an inaccuracy in
the original proof of Theorem~1.2. }

\end{document}